\documentclass[11pt,reqno]{amsart}
\usepackage{amssymb,amsmath}
\usepackage[utf8]{inputenc}
\usepackage{amsthm}
\usepackage{amssymb}
\usepackage{amsmath,amscd}
\usepackage{latexsym}
\usepackage{graphicx,subfigure}
\usepackage[
    hscale=0.70,
    vscale=0.775,
    headheight=13pt,
    centering,
]{geometry}
\usepackage{tikz-cd}
\usepackage{xcolor}
\definecolor{verydarkblue}{rgb}{0,0,0.5}
\usepackage[
    breaklinks,
    colorlinks,
    citecolor=verydarkblue,
    linkcolor=verydarkblue,
    urlcolor=verydarkblue,
    pagebackref=true,
    hyperindex
]{hyperref}

\usepackage[capitalize]{cleveref}
\usepackage{enumitem}

\theoremstyle{plain}

\crefname{introtheorem}{Theorem}{Theorems}

\newtheorem{theorem}{Theorem}

\newtheorem{proposition}[theorem]{Proposition}

\newtheorem{thm}[theorem]{Theorem}

\newtheorem{prop}[theorem]{Proposition}
\newtheorem{cor}[theorem]{Corollary}

\theoremstyle{definition}

\newtheorem{example}[theorem]{Example}
\newtheorem{defn}[theorem]{Definition}

\theoremstyle{remark}

\newtheorem{remark}[theorem]{Remark}

\numberwithin{theorem}{section}
\numberwithin{equation}{section}

\newcommand{\Q}{\mathbb{Q}}
\newcommand{\PP}{\mathbb{P}}

\newcommand{\Z}{\mathbb{Z}}

\newcommand{\G}{\mathbb{G}}

\newcommand{\Aff}{\mathbb{A}}
\newcommand{\aid}{\mathfrak{a}}

\newcommand{\OS}{\mathcal{O}}

\newcommand{\F}{\mathcal{F}}

\newcommand{\LL}{\mathcal{L}}
\newcommand{\M}{\mathcal{M}}
\newcommand{\B}{\mathcal{B}}

\newcommand{\X}{\mathcal{X}}
\newcommand{\Y}{\mathcal{Y}}

\newcommand{\spec}{\operatorname{Spec}}

\newcommand{\proj}{\operatorname{Proj}}

\newcommand{\Supp}{\operatorname{Supp}}

\newcommand{\ord}{\operatorname{ord}}
\newcommand{\vol}{\operatorname{vol}}

\newcommand{\lct}{\operatorname{lct}}

\newcommand{\exc}{\operatorname{Exc}}

\newcommand{\aut}{\operatorname{Aut}}

\newcommand{\cB}{\mathcal{B}}

\newcommand{\cD}{\mathcal{D}}

\newcommand{\cF}{\mathcal{F}}

\newcommand{\cL}{\mathcal{L}}
\newcommand{\cM}{\mathcal{M}}

\newcommand{\cO}{\mathcal{O}}

\newcommand{\cZ}{\mathcal{Z}}
\newcommand{\cX}{\mathcal{X}}
\newcommand{\cY}{\mathcal{Y}}
\newcommand{\cW}{\mathcal{W}}

\newcommand{\fM}{\mathfrak{M}}

\newcommand{\bP}{\mathbb{P}}
\newcommand{\bC}{\mathbb{C}}
\newcommand{\bR}{\mathbb{R}}
\newcommand{\bA}{\mathbb{A}}
\newcommand{\bQ}{\mathbb{Q}}
\newcommand{\bZ}{\mathbb{Z}}

\newcommand{\bG}{\mathbb{G}}

\newcommand{\pr}{\mathrm{pr}}

\newcommand{\ocX}{\overline{\cX}}
\newcommand{\ocL}{\overline{\cL}}
\newcommand{\ocD}{\overline{\cD}}

\DeclareMathOperator{\Aut}{Aut}

\DeclareMathOperator{\Fut}{Fut}

\DeclareMathOperator{\nvol}{\widehat{vol}}

\DeclareMathOperator{\Val}{Val}

\newcommand{\ocY}{\overline{\mathcal{Y}}}
\newcommand{\tD}{\widetilde{D}}
\newcommand{\tB}{\widetilde{B}}
\newcommand{\tcD}{\widetilde{\mathcal{D}}}

\title[Equivariant K-stability Under Finite Group Action]{Equivariant K-stability Under Finite Group Action}
\author{Yuchen Liu}
\address{Department of Mathematics, Yale University, New Haven, CT 06511, USA}
\email{yuchen.liu@yale.edu}

\author{Ziwen Zhu}
\address{Department of Mathematics, University of Utah, Salt Lake City, UT 84112, USA}
\email{{\tt zzhu@math.utah.edu}}
\thanks{Research of ZZ is partially supported by NSF Grant DMS-1265285.}
\date{\today}

\begin{document}
\begin{abstract}
We show that $G$-equivariant K-semistability (resp. $G$-equivariant K-polystability) implies K-semistability (resp. K-polystability) for log Fano pairs when $G$ is a finite group. 
\end{abstract}
\maketitle
\section{introduction}

K-stability is an algebro-geometric notion introduced by Tian \cite{Tia97} and reformulated algebraically by Donaldson \cite{Don02} to detect existence of (conical) K\"ahler-Einstein metrics on Fano varieties or more generally, on log Fano pairs. If a complex log Fano pair $(X,D)$ admits an action by a reductive group $G$, then it is conjectured that to test K-semistability and K-polystability it suffices to test on all $G$-equivariant test configurations. 
When $X$ is smooth and $D=0$, this conjecture is proven by Datar and Sz\'ekelyhidi \cite{DS16}. When $G$ is an algebraic torus, it is proven by Li and Xu \cite{LX16} for K-semistability and Li, Wang, and Xu \cite{LWX18} for K-polystability. 

In this paper, we prove the equivalence of K-semistability (resp. K-polystability) and $G$-equivariant K-semistability (resp. $G$-equivariant K-polystability) for log Fano pairs when $G$ is a finite group. Throughout the paper, we work over the field of complex numbers $\bC$.

\begin{theorem}\label{eq}
Let $(X,D)$ be a log Fano pair. Let $G< \aut(X,D)$ be a finite subgroup. If $(X,D)$ is $G$-equivariantly K-semistable (resp.  $G$-equivariantly K-polystable), then $(X,D)$ is K-semistable (resp. K-polystable).
\end{theorem}

We also show that K-semistability (resp. K-polystability) are preserved under crepant finite surjective morphisms (resp. Galois morphisms) between log Fano pairs.

\begin{thm}\label{thm:finitecover}
Let $\pi:(X,D)\to (Y,B)$ be a finite surjective morphism between log Fano pairs such that $K_X+D=\pi^*(K_Y+B)$. Then
\begin{enumerate}
    \item $(X,D)$ is K-semistable if and only if $(Y,B)$ is K-semistable.
    \item If $\pi$ is a Galois morphism between log Fano pairs (see Definition \ref{defn:galois}), then $(X,D)$ is K-polystable if and only if $(Y,B)$ is K-polystable.
    \item If either $(X,D)$ or $(Y,B)$ is K-unstable, then 
    $\delta(X,D)=\delta(Y,B)$. 
\end{enumerate}
\end{thm}

Note that the ``only if'' direction in Theorem \ref{thm:finitecover}(1) is proven by Fujita \cite[Corollary 1.7]{Fuj19}. We remark here that similar statements to Theorems \ref{eq} and \ref{thm:finitecover} in the K-stable or uniformly K-stable case can fail (see e.g. \cite[Example 4.2]{Fuj19}).

We also generalize Tian's criterion for a finite group $G$ \cite{tian_alpha, Al} to the case of log Fano pairs.

\begin{cor}\label{cor:tian}
Let $(X,D)$ be a complex log Fano pair. Let $G< \aut(X,D)$ be a finite group action on $X$. If $\alpha_G(X,D)> \frac{n}{n+1}$ (resp. $\geq \frac{n}{n+1}$), then $(X,D)$ admits weak conical K\"ahler-Einstein metrics hence is K-polystable (resp. is K-semistable).
\end{cor}

Along the proof, we show that $G$-equivariant K-stability can be tested by $G$-equivariant special test configurations when $G$ is a finite group, generalizing a result of Li and Xu \cite{LX14}.

\begin{thm}\label{Gspec}
Let $(X,D)$ be a log Fano pair. Let $G< \aut(X,D)$ be a finite group action on $X$. If $\Fut(\cX,\cD;\cL)\geq 0$ for all $G$-equivariant special test configurations $(\cX,\cD;\cL)$ of $(X,D)$, then $(X,D)$ is $G$-equivariantly K-semistable.
\end{thm}

As an application, we show that smooth cubic fourfolds of cyclic cover type are K-stable. Such cubic fourfolds were studied by Allcock, Carlson, and Toledo \cite{ACT11} where the authors relate their period map to moduli spaces of cubic threefolds.

\begin{thm}\label{thm:smcubic}
A smooth cubic fourfold $X\subset\bP^5$ with equation $x_5^3=f(x_0,\cdots, x_4)$ is K-stable hence admits K\"ahler-Einstein metrics.
\end{thm}

This paper is organized as follows. In Section \ref{sec:prelim}, we collect preliminaries for K-stability, valuative criterion, stability thresholds, normalized volumes, Koll\'ar components, and K-moduli spaces of log Fano pairs. In Section \ref{sec:special}, we prove Theorem \ref{Gspec} by establishing a correspondence between $G$-equivariant special test configurations of a log Fano pair $(X,D)$ and special test configurations of its $G$-quotient log Fano pair $(Y,B)$. In Section \ref{sec:proofs} we present proofs to Theorems \ref{eq}, \ref{thm:finitecover}, and Corollary \ref{cor:tian}. The K-semistable parts follow from combining the recent work of Li, Tian, and Wang \cite{li2019uniform}, where they confirm the uniform version of Yau-Tian-Donaldson conjecture for log Fano pairs, with a perturbation argument from \cite{BL18b}. Proofs of the K-polystable parts use the recent work of Alper, Blum, Halpern-Leistner, and Xu \cite{ABHLX19} where they show  reductivity of automorphism group of a K-polystable log Fano pair. Finally, we present applications in Section \ref{sec:appl} including the proof of Theorem \ref{thm:smcubic}.

\subsection*{Acknowledgements}
We would like to thank Harold Blum, Tommaso de Fernex, Chi Li, Dapeng Mu, and Chenyang Xu for helpful discussions. We are grateful to Ivan Losev for providing the references \cite{luna1975adherences, VP89} which are essential in the proof of Theorem \ref{eq}.

\section{Preliminaries}\label{sec:prelim}
In this section, we recall the definitions of K-stability and equivariant K-stability. We also include other related notions and results about K-stability that will be used in later sections.
\subsection{Test configurations and K-stability}
\begin{defn}
A log Fano pair $(X,D)$ is a normal projective variety $X$ together with an effective $\bQ$-divisor $D$ such that $-(K_X+D)$ is $\bQ$-Cartier ample, and $(X,D)$ has klt singularities. A group action $G$ on $(X,D)$ is an algebraic subgroup $G<\aut(X,D)$.
\end{defn}

In the following definition, we define $G$-equivariant K-(poly/semi)stability. Note that the usual definition of K-(poly/semi)stability can be recovered by taking $G$ to be the trivial group.

\begin{defn}[\cite{Tia97, Don02, Al, LX14, Li15, OS15}]
Let $(X,D)$ be an $n$-dimensional log Fano pair with an action of a group $G$. Let $L$ be an ample line bundle on $X$ such that $L\sim -l (K_X+D)$ for some $l\in\bZ_{>0}$.
\begin{enumerate}
\item A \emph{$G$-equivariant test configuration}\footnote{In this paper, we only consider test configurations with normal total space.} $(\cX,\cD;\cL)/\bA^1$ of $(X,D;L)$ consists of the following data:
\begin{itemize}
 \item a normal variety $\cX$, an effective $\bQ$-divisor $\cD$ on $\cX$, together with a flat projective morphism $\pi:(\cX,\Supp(\cD))\to \bA^1$;
 \item a $\pi$-ample line bundle $\cL$ on $\cX$;
 \item a $G\times \bG_m$-action on $(\cX,\cD;\cL)$ such that $\pi$ is $G\times\bG_m$-equivariant with respect to the trivial action of $G$ on $\bA^1$ and the standard action of $\bG_m$ on $\bA^1$ via multiplication;
 \item $(\cX\setminus\cX_0,\cD|_{\cX\setminus\cX_0};\cL|_{\cX\setminus\cX_0})$
 is $G\times \bG_m$-equivariantly isomorphic to $(X,D;L)\times(\bA^1\setminus\{0\})$.
\end{itemize}

\item
A $G$-equivariant test configuration is called a \emph{product} test configuration if
\[
(\cX,\cD;\cL)\cong(X\times\bA^1,D\times\bA^1;\pr_1^* L\otimes \pr_2^*\cO_{\bA^1}(k\cdot 0))
\] for some $k\in\bZ$. A product test configuration is called a \emph{trivial} test configuration if the above isomorphism is $G\times\bG_m$-equivariant with respect to the given $G$-action on $X$, trivial $\bG_m$-action on $X$, trivial $G$-action on $\bA^1$, and the standard $\bG_m$-action on $\bA^1$ via multiplication.

A $G$-equivariant test configuration $(\cX,\cD;\cL)$ is called a \emph{$G$-equivariant special test configuration} if $\cL\sim_{\bQ}-l (K_{\cX/\bA^1}+\cD)$ and $(\cX,\cX_0+\cD)$ is plt. In this case, we say that $(X,D)$ \emph{$G$-equivariantly specially degenerates to} $(\cX_0,\cD_0)$ which is necessarily a log Fano pair.\footnote{We sometimes omit the polarization in special test configurations.}

\item (cf. \cite{Wan12, Oda13a}) Assume $\pi:(\cX,\cD;\cL)\to \bA^1$ is a $G$-equivariant test configuration of 
$(X,D;L)$. Let $\bar{\pi}: (\ocX,\ocD;\ocL)\to\bP^1$ be the natural $G\times\bG_m$-equivariant compactification of $\pi$. The \emph{generalized Futaki invariant} of $(\cX,\cD;\cL)$ is defined by the intersection formula
\begin{equation}\label{eq:Futint}
\Fut(\cX,\cD;\cL):=\frac{1}{(-K_X-D)^n}\left(\frac{n}{n+1}\cdot\frac{(\ocL^{n+1})}{l^{n+1}}+\frac{(\ocL^n\cdot (K_{\ocX/\bP^1}+\ocD))}{l^n}\right).
\end{equation}
\item The log Fano pair $(X,D)$ is said to be
\begin{itemize}
  \item \emph{$G$-equivariantly K-semistable} if $\Fut(\cX,\cD;\cL)\geq 0$ for any $G$-equivariant  test configuration $(\cX,\cD;\cL)/\bA^1$ and any $l\in\bQ_{>0}$ such that $L$ is Cartier. 
 \item  \emph{$G$-equivariantly K-stable} if it is $G$-equivariantly K-semistable and $\Fut(\cX,\cD;\cL)=0$ for a $G$-equivariant test configuration $(\cX,\cD;\cL)/\bA^1$ if and only if it is a trivial test configuration.
 \item  \emph{$G$-equivariantly K-polystable} if it is $G$-equivariantly K-semistable and $\Fut(\cX,\cD;\cL)=0$ for a $G$-equivariant test configuration $(\cX,\cD;\cL)/\bA^1$ if and only if it is a product test configuration.
 \item  \emph{$G$-equivariantly K-unstable} if it is not $G$-equivariantly K-semistable.
\end{itemize}
\end{enumerate}
\end{defn}


In order to define uniform K-stability, we first recall the definition of $J^{\textrm{NA}}(\X,\cD;\LL)$ following \cite{Fuj16}. Given a test configuration $(\X,\cD;\LL)$ of $(X,D;L)$, where $L\sim -l(K_X+D)$. let
\begin{center}
\begin{tikzcd}
& \Y \arrow[dl,"p"']
	\arrow[dr,"q"] & \\
X\times \PP^1 \arrow [rr,dotted] & & \overline{\X}.
\end{tikzcd}
\end{center}
be a common normal birational model of $X\times \PP^1$ and $\overline{\X}$. We set
$$
\lambda_{\max}(\X,\cD;\LL):=\frac{p^\ast (\pr_1^* L)^n\cdot q^\ast \overline{\LL}}{L^n},
$$
and define
$$
J^{\textrm{NA}}(\X,\cD;\LL):=\lambda_{\max}(\X,\cD;\LL)-\frac{\overline{\LL}^{n+1}}{(n+1)L^n}
$$
Note that $J^{\textrm{NA}}(\X,\cD;\LL)$ can be viewed as the norm of $(\X,\LL)$, and $J^{\textrm{NA}}(\X,\cD;\LL)=0$ if and only if $(\X,\cD;\LL)$ is a trivial  test configuration.

\begin{defn}
A log Fano pair $(X,D)$ with a $G$-action is \emph{$G$-equivariantly uniformly  K-stable} if there exists $\epsilon\in (0,1)$, such that $\Fut(\cX,\cD;\cL)\geq \epsilon J^{\textrm{NA}}(\X,\cD;\LL)$ for any $G$-equivariant  test configuration $(\cX,\cD;\cL)/\bA^1$ and any $l\in\bQ_{>0}$ such that $L$ is Cartier. 
\end{defn}

\subsection{Valuative criterion and stability thresholds}
We recall the valuative criteria for K-stability developed by Fujita \cite{Fuj16} and Li \cite{Li17}. For simplicity, we only state results for equivariant K-stability under finite group action following \cite{Zhu19}. The original version of valuative criteria can be recovered by taking the group action $G$ to be the trivial one, and we will omit $G$ in the corresponding notations in this case. 

Let $G<\aut (X,D)$ be a finite group action on a log Fano pair $(X,D)$. For any prime divisor $F$ over $X$, let $\pi:~Y\to X$ be an $G$-equivariant proper birational morphism such that $F$ is a prime divisor on $Y$. Denote by $G\cdot F=\sum_i F_i$ the sum of all the prime divisors in the orbit of $F$ under $G$-action. We define
\begin{align*}
S_G(F) & :=\frac{1}{(-K_X-D)^n}\int_0^{+\infty} \vol_Y\left(\pi^\ast(-K_X-D)-xG\cdot F\right)\,dx,\\
\beta^G(F) & :=(-K_X-D)^n (A_{(X,D)}(F)- S_G(F)),
\end{align*}
where $A_{(X,D)}(F)$ is the log discrepancy of $F$ with respect to the pair $(X,D)$. A prime divisor $F$ over $X$ is called $G$-dreamy if the graded ring
$$
\bigoplus_{k,j\geq 0} H^0\left(Y,k\pi^\ast(-K_X-D)-jG\cdot F)\right)
$$
is finitely generated.

We also define 
$$
\tau^G(F):=\sup\{t>0|\vol_X\left(\pi^\ast(-K_X-D)-tG\cdot F\right)>0\}
$$
and
$$
j^G(F)=\int_0^{\tau^G(F)}\bigg(\vol_X(-K_X-D)- \vol_X\left(\pi^\ast(-K_X-D)-tG\cdot F\right)\bigg)\,dx.
$$

\begin{theorem}\cite[Theorem A]{Zhu19}\label{val}
Under the above notation, we have the following characterization of K-stability in terms of $\beta^G$:
\begin{enumerate}
\item The following are equivalent:
 	\begin{enumerate}
 		\item[(i)] $(X,D)$ is  $G$-equivariantly uniformly K-stable;
 		\item[(ii)] there exists $0<\epsilon<1$, such that $\beta^G(F)\geq \epsilon j^G(F)$ for any prime divisor $F$ over $X$;
 		\item [(iii)] there exists $0<\epsilon<1$, such that $\beta^G(F)\geq \epsilon j^G(F)$ for any $G$-dreamy prime divisor $F$ over $X$.
 	\end{enumerate}
 \item The following are equivalent:
 	\begin{enumerate}
 		\item[(i)] $(X,D)$ is $G$-equivariantly K-semistable;
 		\item[(ii)] $\beta^G(F)\geq 0$ for any prime divisor $F$ over $X$;
 		\item [(iii)] $\beta^G(F)\geq 0$ for any $G$-dreamy prime divisor $F$ over $X$.
 	\end{enumerate}
 \item The following are equivalent:
 	\begin{enumerate}
 		\item[(i)] $(X,D)$ is $G$-equivariantly K-stable;
 		\item[(ii)] $\beta^G(F)> 0$ for any $G$-dreamy prime divisor $F$ 						over $X$.
 	\end{enumerate}
\end{enumerate}
\end{theorem}

\begin{defn}
Following \cite{BJ17}, we define the $G$-equivariant stability threshold of a log Fano pair $(X,D)$ as
\begin{equation}\label{deltadef}
\delta_G(X,D):=\inf_{F}\frac{A_{(X,D)}(F)}{S_G(F)},    
\end{equation}
where $F$ runs through all prime divisors over $X$.
When $G$ is the identity group, we simply denote $S_G(F)$ by $S(F)$ and $\delta_G(X,D)$ by $\delta(X,D)$. 
\end{defn}
Note that the original definition of $\delta(X,D)$ introduced by Fujita and Odaka in \cite{FO16} involves the log canonical thresholds of basis type divisors. By the work of Blum and Jonsson in \cite{BJ17}, the two definitions (when $G$ is trivial) are equivalent. By similar arguments in the proof of Proposition 3.11 in \cite{BJ17} when $G$ is trivial (See also \cite[Proof of Theorem 5.1]{Zhu19}), for any prime divisor $F$ over $X$, we also have 
$$
\frac{1}{n+1}\tau^G(F)\leq S_G(F)\leq \frac{n}{n+1}\tau^G(F),    
$$
and hence
$$
\frac{S_G(F)}{n}\leq \frac{j^G(F)}{\vol_X(-K_X-D)}\leq nS_G(F).
$$
Then using Blum-Jonsson's definition, we see that the name stability threshold comes from the following immediate consequence of Theorem \ref{val}:
\begin{cor}
Under the above notation, we have
\begin{enumerate}
    \item $(X,D)$ is $G$-equivariantly K-semistable if and only if $\delta_G(X,D)\geq 1$;
    \item $(X,D)$ is $G$-equivariantly uniformly K-stable if and only if $\delta_G(X,D)> 1$.
\end{enumerate}
\end{cor}

\begin{remark}
In \cite{golota2019delta}, Golota defines equivariant stability threshold for connected algebraic group action by taking the infimum in \eqref{deltadef} among invariant prime divisors over $X$, although for finite group action, they use invariant basis type divisors instead. 
It is an interesting question whether our $\delta_G$ coincides with Golota's definition when $G$ is a finite group. 
\end{remark}

\subsection{Normalized volumes and Koll\'ar components}
Let $x\in (X,D)$ be an $n$-dimensional klt singularity. Denote by $\Val_{X,x}$ the $\bR$-valued valuations on $\bC(X)$ centered at $x$. For any valuation $v\in \Val_{X,x}$, the volume of $v$ is defined as
$$
\vol(v)=\limsup_{m\to \infty}\frac{l\left(\OS_{X,x}/(\aid_m(v)\cdot\OS_{X,x})\right)}{m^n/n!},
$$
where $\aid_m(v)$ is the ideal sheaf of regular functions with vanishing order at least $m$ with respect to the valuation $v$, and $l$ is the length. The normalized volume function $\nvol$ introduced by Li in \cite{Li18} is defined as
$$
\nvol_{(X,D),x}(v):=\begin{cases}
A_{(X,D)}(v)^n\vol(v) & \textrm{if }A_{(X,D)}(v)<+\infty\\
+\infty & \textrm{otherwise}
\end{cases}
$$
where $A_{(X,D)}(v)$ is the log discrepancy of $v$. It is proven by Blum \cite{Blu18} that there always exists a valuation $v_{\min}$ which minimizes $\nvol$.
Note that for klt singularities, such minimizing normalized volume $\nvol_{(X,D),x}(v_{\min})$ (also called the local volume of $x\in (X,D)$) is always positive by \cite{Li18, Blu18}. 

Next we define Koll\'ar components over a klt singularity $x\in (X,D)$ following \cite{LX16}.
\begin{defn}
A prime divisor $S$ over a klt closed point $x\in (X,D)$ is called a \emph{Koll\'ar component} if there exists a proper birational morphism $\pi:~Y\to X$ such that $S=\exc(\pi)$ is centered at $x$, the pair $(Y,S+\pi^{-1}_\ast D)$ is plt and $-S$ is ample over $X$.
\end{defn}
A  Koll\'ar component $S$ over $x\in (X,D)$ defines a divisorial valuation $
\ord_S\in \Val_{X,x}$. Therefore we can consider the normalized volume $\nvol_{(X,D),x}(\ord_S)$. It is proven by Li and Xu \cite{LX16} that there exists a sequence $S_k$ of Koll\'ar components such that $\nvol_{(X,D),x}(\ord_{S_k})$ converges to  $\nvol_{(X,D),x}(v_{\min})$ as $k\to\infty$.
Please refer to \cite{LX16, LLX18} for more details about the relation between normalized volume, Koll\'ar components, and K-semistability.

\subsection{Log Fano K-moduli spaces}\label{sec:goodmoduli}

In this section, we collect useful results from recent progress on the algebro-geometric construction of  K-moduli spaces of log Fano pairs.

\begin{defn}
Let $n\in\bZ_{>0}$, $V\in\bQ_{>0}$, and $I\subset [0,1]\cap \bQ$ be a finite subset that is closed under addition\footnote{A subset of $[0,1]$ is closed under addition if any finite sum of elements either belongs to this set or is greater than $1$.}. 
We define the moduli pseudo-functor $\fM_{n,I,V}^{\rm Kss, sn}$ over seminormal schemes $S$ of finite type over $\bC$ as 
\[
\fM_{n,I,V}^{\rm Kss, sn}(S)=\left\{\begin{array}{l}\textrm{Proper locally stable families $(X,D)\to S$ whose geometric fibers}\\
\textrm{are K-semistable log Fano pairs of dimension $n$ and volume $V$}\\
\textrm{with coefficients in $I$.}
\end{array}
\right\}
\]
Here \emph{locally stable} is considered in the sense of \cite[Definition-Theorem 4.45]{Kol17}.
\end{defn}

The following theorem follows from recent developments in the algebro-geometric construction of log Fano K-moduli spaces. 

\begin{thm}[\cite{Jia17, BX18, ABHLX19, BLX19, Xu19}]\label{thm:goodkmoduli}
The moduli pseudo-functor $\fM_{n,I,V}^{\rm Kss, sn}$ is represented by a seminormal Artin stack $\cM_{n,I,V}^{\rm Kss, sn}$ of finite type over $\bC$. Moreover, $\cM_{n,I,V}^{\rm Kss, sn}$ admits a separated good moduli space $M_{n,I,V}^{\rm Kps, sn}$ whose closed points parametrize K-polystable log Fano pairs of dimension $n$ and volume $V$ with coefficients in $I$.
\end{thm}

\begin{proof}
Boundedness follows from \cite{Jia17} and its generalization to log Fano pairs (see \cite{Che18} and \cite[Corollary 6.14]{LLX18}). Openness follows from \cite{BLX19, Xu19}. Separatedness follows from \cite{BX18}. Existence of good moduli spaces follows from \cite{ABHLX19}. 
\end{proof}

The following corollary follows from Theorem \ref{thm:goodkmoduli}, \cite[Proposition 4.4]{AHLH18}, and \cite[Remark 2.11]{ABHLX19}. 

\begin{cor}[loc. cit.]\label{cor:localquot} With the above notation, let $(X,D)$ be a K-polystable pair representing a closed point $x:=[(X,D)]$ in $\cM:=\cM_{n,I,V}^{\rm Kss, sn}$. Denote by $G_x$ the stabilizer group of $x$ in $\cM$. Then there exists an \'etale morphism $f: (\cW,w)\to (\cM,x)$ where $\cW=[\spec(A)/G_x]$ such that 
\begin{enumerate}
    \item $f$ induces an isomorphism of stablizer groups at all points;
    \item $f$ sends closed points to closed points;
    \item $f$ is $\Theta$-surjective (see \cite[Section 3.4]{AHLH18} for a definition). 
\end{enumerate}
\end{cor}

\section{Special test configurations and equivariant K-stability}\label{sec:special}

In this section, we consider the relation between $G$-equivariant K-semistability and K-semistability of the quotient. As a consequence, we prove Theorem \ref{Gspec} which states that in order to check $G$-equivariant K-semistability, it is enough to examine only $G$-equivariant special test configurations.

We first fix some notation in this section. Let $G$ be a finite group action on a log Fano pair $(X,D)$, and $\sigma:~X\to Y:=X/G$ the quotient map. By the canonical bundle formula, we have 
$$
K_X+D=\sigma^\ast (K_Y+D')+\sum (e_R-1)R,
$$
where $D'$ is the $\bQ$-divisor on $Y$ such that $\sigma^\ast D'=D$, and $R$ runs through all ramification divisors and $e_R$ is the ramification index of $R$.

Now let $\sum B_i$ be the collection of irreducible branch divisors on $Y$. Then we have
$$
\sigma^\ast B_i=\sum_{\sigma(R)=B_i}e_R R.
$$
Note that for each $B_i$, the relevant ramification divisors with $\sigma(R)=B_i$ all have the same ramification index since they belong to the same $G$-orbit, and we denote the common $e_R$ by $e_i$. Let
$$
B:=D'+\sum_i \left(1-\frac{1}{e_i}\right)B_i.
$$
Then we have $K_X+D=\sigma^\ast (K_Y+B)$, and the pair $(Y,B)$ is log Fano by \cite[Proposition 5.20]{KM98}. 

\begin{defn}\label{defn:galois}
We say a finite surjective morphism $\sigma: (X,D)\to (Y,B)$ between log Fano pairs is a \emph{Galois morphism} if there exists a finite group $G$ acts on $(X,D)$ such that $Y\cong X/G$ and $B$ is the $\bQ$-divisor determined above such that $K_X+D=\sigma^*(K_Y+B)$.
\end{defn}

\begin{example}\label{ex_fujita}
Consider the Galois morphism $\pi:~\left(\PP^1,d[1]+d[-1]\right)\to \left(\PP^1,\frac{1}{2}[0]+\frac{1}{2}[\infty]+d[1]\right)$ in \cite[Example 4.2]{Fuj19},
where $\pi$ can be realized by quotient under the $G=\Z/2$-action on $\left(\PP^1,d[1]+d[-1]\right)$ with $[t]\mapsto [-t]$. For $d\in (0,1)$, we know that $\left(\PP^1,\frac{1}{2}[0]+\frac{1}{2}[\infty]+d[1]\right)$ is uniformly K-stable while $\left(\PP^1,d[1]+d[-1]\right)$ is not,  but $\left(\PP^1,d[1]+d[-1]\right)$ is in fact  $G$-equivariantly uniformly K-stable. Indeed, by direct computation of stabiliy thresholds using definition, we have
$$
\delta_G\left(\PP^1,d[1]+d[-1]\right)=\delta\left(\PP^1,\frac{1}{2}[0]+\frac{1}{2}[\infty],d[1]\right)=\min\left\{2,\frac{1}{1-d}\right\}>1.
$$
\end{example}

\begin{example}\label{ex_triangle}
Let $\Delta(l,m,n)$ be a triangle group and assume $l\leq m\leq n$. For suitable choice of $l,m,n$ with $1/l+1/m+1/n>1$, $\Delta(l,m,n)$ is a finite group generated by reflections across the sides of a triangle on a real 2-sphere with angles $\pi/l,\pi/m,\pi/n$. Denote by $a,b,c$ the three reflections generating $\Delta(l,m,n)$. Consider the subgroup $D(l,m,n)< \Delta(l,m,n)$ generated by $ab$, $bc$ and $ca$, which are  rotations centered at the vertices of the triangle by the angles of $2\pi/l$, $2\pi/m$ and $2\pi/n$ respectively. Note that $D(l,m,n)$ can be considered as an algebraic group action on $\PP^1$ without any fixed point. Denote by $G$ the representation of $D(l,m,n)$ inside $\aut(\PP^1)=\mathrm{PGL}(2,\bC)$. By taking the quotient under $G$-action, we get a log Fano pair 
$$
\left(\PP^1,\frac{l-1}{l}p_1,\frac{m-1}{m}p_2,\frac{n-1}{n}p_3\right),
$$ 
where $p_1$, $p_2$ and $p_3$ are the three fixed points of the three rotations respectively. By direct computation of stability thresholds using definitions, we have
$$
\delta_G(\PP^1)=\delta\left(\PP^1,\frac{l-1}{l}p_1,\frac{m-1}{m}p_2,\frac{n-1}{n}p_3\right)=\frac{2/n}{1/l+1/m+1/n-1}\geq 2.
$$
Therefore we know that the pair $\left(\PP^1,\frac{l-1}{l}p_1,\frac{m-1}{m}p_2,\frac{n-1}{n}p_3\right)$ is uniformly K-stable and $\PP^1$ is $G$-equivariantly uniformly K-stable. 
\end{example}

Next, under the above notation, we show the equivalence between various $G$-equivariant K-stability conditions of $(X,D)$ and corresponding K-stability conditions of its quotient $(Y,B)$.

\begin{proposition}\label{quotient}
Under the above notation, we have the following:
\begin{enumerate}
    \item $(X,D)$ is $G$-equivariantly K-semistable  if and only if $(Y,B)$ is K-semistable.
    \item $(X,D)$ is $G$-equivariantly (uniformly) K-stable  if and only if $(Y,B)$ is (uniformly) K-stable.
\end{enumerate}

\end{proposition}

\begin{proof}
(1) 
For the ``only if'' part, We mostly follow the same idea as in   \cite[Proof of Corollary 1.7]{Fuj19}. Assume $(X,D)$ is  $G$-equivariantly K-semistable.  For any divisor $F$ over $Y$, it suffices to show that $\beta(F)\geq 0$. Let $f:~Y'\to Y$ be a birational morphism such that $F$ is a divisor on $Y'$. Consider the following diagram
\begin{center}
\begin{tikzcd}
X'\arrow[d,"g"]
 \arrow[r,"\sigma'"] & Y' \arrow[d,"f"]\\
 X \arrow[r,"\sigma"] & Y,
\end{tikzcd}
\end{center}
where $X'$ is the normalization of the fiber product $X\times_Y Y'$. Note that there is a natural $G$-action on $X\times_Y Y'$ and hence on $X'$, such that $\sigma'$ is the quotient map. Let
$$
\sigma'^\ast F=\sum_{i=1}^m e_{F_i'}F_i',
$$
where $e_{F_i'}$ is the ramification index of $F_i'$. Since $\sigma'$ is the quotient map of the $G$-action, all the $e_{F_i'}$'s are the same which we denote by $e_F$. By \cite[Proof of Proposition 5.20]{kollar2008birational}, we have the following relation between log discrepancies
$$
e_F A_{(Y,B)}(F)=A_{(X,D)}(F_i').
$$
Also note that
$$
\sigma'^\ast (f^\ast (-K_Y-B)-xF)=g^\ast (-K_X-D)-x e_F\sum_{i=1}^m F_i',
$$
and by \cite[Lemma 4.1]{Fuj19}, we know that
$$
\vol(\sigma'^\ast (f^\ast (-K_Y-B)-xF)\geq \deg(\sigma') \vol(f^\ast (-K_Y-B)-xF),
$$
where $\deg(\sigma')=|G|$.
Therefore, we have
\begin{align*}
\beta(F)&=A_{(Y,B)}(F)(-K_Y-B)^n-\int_0^\infty\vol(f^\ast (-K_Y-B)-xF)\,dx\\
&\geq \frac{1}{|G|}\left(A_{(Y,B)}(F)(-K_X-D)^n-\int_0^\infty\vol\left(\sigma'^\ast (f^\ast (-K_Y-B)-xF)\right)\,dx\right)\\
&= \frac{1}{e_F|G|}\left(A_{(X,D)}(F_i')(-K_X-D)^n-\int_0^\infty\vol\left(g^\ast (-K_X-D)-x\sum_{i=1}^m F_i'\right)\,dx\right)\\
&=\frac{\beta^G(F_1')}{e_F|G|}\geq 0,
\end{align*}
where the last inequality follows from  \ref{val}. This finishes the proof of the ``only if'' part.


Next we treat the ``if'' part. 
Assume $(Y,B)$ is K-semistable. Let $(\X,\cD;\LL)$ be any $G$-equivariant test configuration of $(X,D)$. Since $G$ is finite, we may take $r:=|G|$ such that any stabilizer subgroup $G_x<G$ of $x\in X$ acts trivially on $\cL_x^{\otimes r}$. By taking the quotient map $\sigma_{\cX}:~\X\to \Y:=\X/G$, we get a test configuration $(\Y,\B;\M)$ of $(Y,B)$, where $\B$ is the Zariski closure of $B$ on the general fibers such that $K_\X+\cD=\sigma_{\cX}^\ast (K_\Y+\B)$, and $\M$ is the descent line bundle of $\cL^{\otimes r}$ to $\Y$.

Note that after compactifying the test configurations over $\PP^1$, we have a quotient map $\sigma_{\ocX}:\ocX\to\ocY$ which naturally extends $\sigma_{\cX}$. Moreover, we have
\begin{align*}
K_X+D=\sigma^\ast(K_Y+B),\qquad
\overline{\LL}^{\otimes r}=\sigma_{\ocX}^\ast \overline{\M},\qquad
K_{\overline{\X}/\PP^1}+\cD=\sigma_{\ocX}^\ast (K_{\overline{\Y}/\PP^1}+\B).
\end{align*}
By the intersection formula of generalized Futaki invariants \eqref{eq:Futint} and K-semistability of $(Y,B)$, we have 
$$
0\leq \Fut(\Y,\B;\M)=\Fut(\X,\cD;\LL^{\otimes r})=\Fut(\cX,\cD;\cL).
$$
Thus $(X,D)$ is $G$-equivariantly K-semistable. This finishes the proof of the ``if'' part. Hence we finish the proof of part (1). 
\medskip


(2) We first treat the ``only if'' part. Indeed, from the proof of the``only if'' part of part (1), for any prime divisor $F$ over $Y$, we have
$$
S_G(F_1')\geq e_F S(F).
$$
Therefore, if $(X,D)$ is $G$-equivariantly (uniformly) K-stable, then $(Y,B)$ is (uniformly) K-stable. 

For the ``if'' part, 
take $G$-dreamy prime divisor $F_1'$ over $(X,D)$ and let $F_1'\ldots,F_m'$ form the orbit of $F_1'$ under $G$-action. Let $F$ be the image of $F_i'$'s under the quotient map $\sigma$, and $e_F$ the ramification index of all $F_i'$'s. Consider the $G$-invariant filtration on $R_j=H^0(X,-j(K_X+D))$ determined by $F_1'$:
$$
\F_{(X,D)}^xR_j=H^0\left(X',g^\ast (-jK_X-jD)-xe_F\sum_{i=1}^m F_i'\right).
$$
If we take $G$-invariant sections of $\F_{(X,D)}^x R_j$, we get sections of the following filtration on $V_j=H^0(Y,-j(K_Y+B))$ determined by $F$:
$$
\F_{(Y,B)}^xV_j=H^0\left(Y',f^\ast (-jK_Y-jB)-xF\right).
$$
Without loss of generality, we may assume the above two filtrations are finitely generated in degree one of the grading $j$.
Then taking the quotient under the $G$-action of the $G$-equivariant test configuration of $(X,D)$
$$
\left(\proj_{\Aff^1}t^{-p}\F_{(X,D)}^pR_j,\cD;\OS(1)\right),
$$
we get the test configuration of $(Y,B)$ induced by $F$:
$$
\left(\proj_{\Aff^1}t^{-p}\F_{(Y,B)}^pV_j,\cB;\OS(1)\right),
$$
The same argument in the proof of the ``if'' part of part (1) gives the equality of generalized Futaki invariants of the above two test configurations. From the computation of generalized Futaki invariants of test configurations induced by dreamy divisors in \cite{Fuj16} (and in \cite{Zhu19} for the $G$-equivariant case), we know that $\beta^G(F_1')=e_F|G|\beta(F)$. Note that since $e_F A_{(Y,B)}(F)=A_{(X,D)}(F_i')$, then we also have $S_G(F_1')= e_F S(F)$. Then by the valuative criteria in Theorem \ref{val}, we know that if $(Y,B)$ is (uniformly) K-stable, then $(X,D)$ is $G$-equivariantly (uniformly) K-stable. 
\end{proof}

\begin{remark}
Since $S_G(F_1')\geq e_F S(F)$, we know that $\delta(Y,B)\geq \delta_G(X,D)$. When either $\delta_G(X,D)$ or $\delta(Y,B)$ can be approximated by $G$-dreamy divisors over $X$ or dreamy divisors over $Y$ respectively (for example, the latter is true when $\delta(Y,B)\leq 1$ by \cite[Theorem 4.3]{BLZ19}), then by the argument  in the proof of Proposition \ref{quotient}(2), we have $\delta(Y,B)= \delta_G(X,D)$, as is illustrated in Example \ref{ex_fujita} and Example \ref{ex_triangle}.
\end{remark}

\begin{remark}
Note that the similar statement to Proposition \ref{quotient} in the K-polystable case is implied by Theorem \ref{thm:finitecover}(2) which is proven in Section \ref{sec:proofs}. Our proof uses two very recent results: the solution of uniform Yau-Tian-Donaldson conjecture \cite{li2019uniform}, and the reductivity of automorphism groups of K-polystable log Fano pairs \cite{ABHLX19}.
\end{remark}


Next we discuss a criterion for $G$-equivariant K-semistability in terms of normalized volumes.
Choose a positive integer $l$ such that $L:=-l(K_X+D)$ is an ample Cartier divisor on $X$ where each stabilizer group $G_x$ acts trivially on $L_x$ for any $x\in X$. Let $Z:=C(X,L)$ be the affine cone over $X$ with polarization $L$. Denote by $o\in Z$ the cone vertex. Let $\tD$ be the $\bQ$-divisor on $Z$ as the Zariski closure of the pull back of $D$ under the projection $Z\setminus\{o\}\to X$. Then from \cite[Section 3.1]{Kol13} we know that $(Z,\tD)$ is a klt singularity. Let $v_0$ be the canonical valuation corresponding to the exceptional divisor of the blow-up of $Z$ at $o$.

\begin{proposition}\label{Kc}
Under the above notation, $(X,D)$ is $G$-equivariantly K-semistable if and only if the normalized volume $\widehat{\vol}_{(Z,\tD),o}$ is minimized at the canonical valuation $v_0$ among all $G\times \G_m$-invariant Koll\'ar components over $o\in (Z,\tD)$.
\end{proposition}
\begin{proof}
Note that there is a natural $G$-action on the cone $Z$ induced from the $G$-action on $(X,D)$ and the $G$-linearization of $L$. Taking the quotient map, we get $\tilde{\sigma}:~Z\to W:=Z/G$, where $W=C(Y,M)$ is the affine cone over $Y$ with the polarization $M$ as the descent ample line bundle of $L$ to $Y$. Let $o'$ be the vertex of $W$ and $v'_0$ be the canonical valuation on $W$ corresponding to the exceptional divisor of the blow-up at $o'$. Let $\tB$ be the $\bQ$-divisor on $W$ as the Zariski closure of the pull back of $B$ under the projection $W\setminus\{o'\}\to Y$. Then it is clear that $K_Z+\tD=\tilde{\sigma}^*(K_W+\tB)$.

By \cite[Lemma 2.13]{LX16}, we know that there is a one-to-one correspondence via pullback between $G\times\bG_m$-invariant Koll\'ar components $S$ over $o\in (Z,\tD)$ and $\bG_m$-invariant Koll\'ar components $S'$ over $o'\in (W,\tB)$.
By \cite[Lemma 2.14]{LX16}, we have
$$
|G|\cdot\widehat{\vol}_{(W,\tB),o'}(\ord_{S'})=\widehat{\vol}_{(Z,\tD),o}(\ord_S),
$$
for any $\bG_m$-invariant Koll\'ar component $S'$ over $o'\in (W, \tB)$ and $S$ the pullback of $S'$ to $o\in (Z,\tD)$. It is clear that the pull back of $v_0'$ is precisely $v_0$.
Therefore, the proposition follows from \cite[Proposition 4.4 and Theorem 4.5]{LX16} together with Proposition \ref{quotient}.
\end{proof}

\begin{prop}\label{prop:specialquot}
Under the above notation, any $G$-equivariant special test configuration $(\cX,\cD)$ of $(X,D)$ produces a special test configuration $(\cY,\cB)$ of $(Y,B)$ by taking quotient of $G$-action. Two $G$-equivariant special test configurations of $(X,D)$ have isomorphic quotients if and only if they are isomorphic.
Conversely, for any special test configuration $(\cY,\cB)$ of $(Y,B)$, there exists a positive integer $d$ such that the test configuration $(\cY^{(d)},\cB^{(d)}):= (\cY,\cB)\times_{\bA^1}\bA^1$ as a base change under the map $\bA^1\to \bA^1$ by $t\mapsto t^d$ is obtained this way.
\end{prop}

\begin{proof}
For simplicity we omit the polarization for special test configurations.
Suppose $(\cX,\cD)$ is a $G$-equivariant special test configuration of $(X,D)$. Then by taking quotient of $G$, we obtain a test configuration $(\cY,\cB)$ of $(Y,B)$ where $\sigma_{\cX}:\cX\to \cY=\cX/G$ is the quotient map and $\cB$ satisfies $K_{\cX}+\cD=\sigma_{\cX}^*(K_{\cY}+\cB)$. Since $(\cX,\cD+\cX_0)$ is plt,  by \cite[Proposition 5.20]{KM98} we know that $(\cY,\cB+\cY_0)$ is also plt which implies that $(\cY,\cB)$ is special.

Next we prove the injectivity of this quotient construction. Suppose $(\cX,\cD)$ and $(\cX',\cD')$ are two $G$-equivariant special test configurations of $(X,D)$ with isomorphic quotient $(\cY,\cB)\cong (\cY',\cB')$.  Consider the two $G$-invariant valuations $v:=\ord_{\cX_0}|_{\bC(X)}$ and $v':=\ord_{\cX_0'}|_{\bC(X)}$. Then it is clear that 
\[
v|_{\bC(Y)}=\ord_{\cY_0}|_{\bC(Y)}=\ord_{\cY_0'}|_{\bC(Y)}=v'|_{\bC(Y)}.
\]
Since $\bC(Y)=\bC(X)^G$, for any rational function $f\in \bC(X)$ we have that $\Pi_{g\in G}f\circ g \in \bC(Y)$. Hence
\[
v(f)=\frac{1}{|G|}v (\Pi_{g\in G}f\circ g)=\frac{1}{|G|}v' (\Pi_{g\in G}f\circ g)=v'(f).
\]
Thus $v=v'$ which implies that $(\cX,\cD)\cong (\cX',\cD')$ by \cite[Lemma 5.17]{BHJ17}.

Conversely, let $(\cY,\cB)$ be a non-trivial special test configuration of $(Y,B)$. We will construct $(\cX,\cD)$ using the cone construction from \cite{Li17, LWX18}. By \cite[Lemma 4.1]{BHJ17}, we have
$\ord_{\cY_0}|_{\bC(Y)}=b\cdot \ord_F$
for some prime divisor $F$ over $Y$ and $b\in \bZ_{>0}$. Recall that $W=C(Y,M)$ is the affine cone over $Y$ with  $M=-l(K_Y+B)$ an ample Cartier divisor. Let $Y_0$ be the exceptional divisor of the canonical blow-up of $W$ at the vertex $o'$. Let $F_{\infty}$ be the prime divisor over $W$ as the pull-back of $F$ over $Y$ under the projection $W\setminus \{o'\}\to Y$. For $k\gg 1$, let $w_k$ be the quasi-monomial valuation on $W$ of weights $(1-\frac{l b A_{Y,B}(F)}{k}, \frac{b}{k})$ along $(Y_0, F_\infty)$. For simplicity, assume that $k$ is coprime to $b$. Then by \cite[Lemma 3.4]{LWX18} we know that $kw_k=\ord_{S_k'}$ where $S_k'$ is a $\bG_m$-invariant Koll\'ar component over $o'\in (W, \tB)$. Pulling back $S_k'$ to $(Z,\tD)$, by \cite[Lemma 2.13]{LX16} we obtain a $G\times \bG_m$-invariant Koll\'ar component $S_k$ over $(Z,\tD)$ such that $\ord_{S_k}|_{\bC(W)}=d_k\cdot\ord_{S_k'}$ for some $d_k\in \bZ_{>0}$. Note that $Z=\spec(\oplus_{j=0}^\infty R_j)$ with $R_j:=H^0(X,jL)$. Then the valuation $\ord_{S_k}$ defines a $\bZ$-filtration on $R_\bullet$ by
$$
\F^p R_j=\{f\in R_j\mid \ord_{S_k}(f)\geq p\}.
$$
By taking the Rees construction, we denote 
\begin{equation}\label{eq:cZ}
 \cZ:= \spec_{\bC[t]} \bigoplus_{p\in\bZ}\bigoplus_{j\in\bZ_{\geq 0}}t^{-p}\cF^{p}R_j.
\end{equation}
It is clear that $\cZ\setminus\cZ_0\cong Z\times(\bA^{1}\setminus\{0\})$. 
Let $\tcD$ be the Zariski closure of $\tD\times(\bA^1\setminus\{0\})$ in $\cZ$. 
Then from \cite[Section 3]{LWX18} we know that $(\cZ,\tcD,\xi;\eta)\to\bA^1$ provides a $G$-equivariant 
special test configuration of the log Fano cone $(Z,\tD,\xi)$ where $\xi$ corresponds to the grading of $k$ and 
$\eta:=t\partial_t$. 
By taking $\proj$ from \eqref{eq:cZ} with respect to the grading of $j$, we obtain a $G$-equivariant special test configuration
$(\cX,\cD)$ of $(X,D)$. 
From the above construction, we have $\ord_{S_k}|_{\bC(X)}=\ord_{\cX_0}|_{\bC(X)}$. Denote by $(\cY',\cB')$ the quotient of $(\cX,\cD)$ by the $G$-action. Then we have 
\[
\ord_{\cY_0'}|_{\bC(Y)}=\ord_{\cX_0}|_{\bC(Y)}=d_k\ord_{S_k'}|_{\bC(Y)}=kd_k w_k|_{\bC(Y)}=d_k b\cdot \ord_F=d_k\cdot \ord_{\cY_0}|_{\bC(Y)}.
\]
Thus by \cite[Lemma 5.17]{BHJ17} we know that $(\cY',\cB')\cong (\cY^{(d_k)}, \cB^{(d_k)})$. The proof is finished. 
\end{proof}

Now we would like to prove Theorem \ref{Gspec}.
\begin{proof}[Proof of Theorem \ref{Gspec}]
Assume to the contrary that $(X,D)$ is $G$-equivariantly K-unstable. By Proposition \ref{quotient}, we know that $(Y,B)$ is K-unstable. Hence by \cite{LX14} there exists a special test configuration $(\cY,\cB)$ of $(Y,B)$ such that $\Fut(\cY,\cB)<0$. 
Then Proposition \ref{prop:specialquot} implies that there exists a $\bG$-equivariant special test configuration $(\cX,\cD)$ of $(X,D)$ and $d\in \bZ_{>0}$ such that $(\cY^{(d)},\cB^{(d)})$ is obtained by quotient of the $G$-action on $(\cX,\cD)$.
By the intersection formula of generalized Futaki invariants \eqref{eq:Futint}, we have 
\[
\Fut(\cX,\cD)=\Fut(\cY^{(d)},\cB^{(d)})=d\cdot \Fut(\cY,\cB)<0.
\]
This is a contradiction.
\end{proof}

\section{Proof of main theorems}\label{sec:proofs}
In this section, we prove our main results Theorems \ref{eq}, \ref{thm:finitecover}, and Corollary \ref{cor:tian}.

The following theorem in \cite{BL18b} allows us to perturb boundaries for a K-semistable Fano pair, and we will use it repeatedly throughout the section. Here we state the version of the theorem for log Fano pairs.
\begin{theorem}[{\cite[Theorem C]{BL18b}}]\label{perturb}
Let $(X,D)$ be a log Fano pair. We have
\begin{align*}
&\min\{1, \delta(X,D)\}\\
=&\sup\{ \beta\in (0,1]|(X,D+(1-\beta )\Delta)~\text{is K-semistable for some}~\Delta\in |-K_X-D|_\Q\}\\
=&\sup\{ \beta\in (0,1]|(X,D+(1-\beta) \Delta)~\text{is uniformly K-stable for some}~\Delta\in |-K_X-D|_\Q\}.
\end{align*}
\end{theorem}

\begin{proof}[Proof of Theorem \ref{eq}]
We first prove the K-semistable part.
Assume that the log Fano pair $(X,D)$ is $G$-equivariantly K-semistable. Using notation as before, we consider the quotient map of $(X,D)$ under $G$ as $\sigma:~(X,D)\to (Y,B)$ where $(Y,B)$ is also a log Fano pair. By Proposition \ref{quotient}, we know that $(Y,B)$ is K-semistable. Then by Theorem \ref{perturb}, for any $\epsilon> 0$, we can find $\Delta'\in |-(K_Y+B)|_\Q$ such that $(Y,B+\epsilon \Delta')$ is uniformly K-stable. By \cite[Corollary 3.5]{BX18}, we know that $(Y,B+\epsilon \Delta')$ has finite automorphism group. Hence by \cite[Corollary 1.2]{li2019uniform} we know that $(Y,B+\epsilon \Delta')$ is K\"{a}hler-Einstein. Pulling back the K\"{a}hler-Einstein metric on $(Y,B+\epsilon \Delta')$ to $X$, we get that $(X,D+\epsilon \Delta)$ is K\"{a}hler-Einstein where $\Delta:=\sigma^\ast \Delta'$. Therefore by \cite{Berman}, we see that $(X,D+\epsilon \Delta)$ is K-semistable.  Then using Theorem \ref{perturb} again, we have $\delta(X,D)\geq 1$ which implies that $(X,D)$ is K-semistable.

For the K-polystable part, we will use the results from Section \ref{sec:goodmoduli}. Assume to the contrary that the log Fano pair $(X,D)$ is $G$-equivariantly K-polystable but not K-polystable. For simplicity, assume $G<
\Aut(X,D)$. By the above argument, we know that $(X,D)$ is K-semistable but not K-polystable. Hence by \cite{LWX18} there exists a special test configuration $(\cX,\cD;\cL)/\bA^1$ of $(X,D)$ such that the central fiber $(\cX_0,\cD_0)$ is K-polystable and not isomorphic to $(X,D)$. Let $\cM:=\cM_{n,I,V}^{\rm Kss, sn}$ be the corresponding K-moduli stack of log Fano pairs containing $x:=[(X,D)]$ and $x_0:=[(\cX_0,\cD_0)]$. Denote by $H:= \Aut(\cX_0,\cD_0)$. 
By Corollary \ref{cor:localquot}, there exists an \'etale morphism $f: (\cW,w_0)\to (\cM,x_0)$ where $\cW=[\spec(A)/H]$ such that $f$ induces an isomorphism of stabilizer groups at all points, and $f$ sends closed points to closed points. 
The special test configuration $(\cX,\cD;\cL)/\bA^1$ induces a morphism $\phi: \Theta\to \cM$ where $\Theta:=[\bA^1/\bG_m]$ such that $\phi(1)=x$ and $\phi(0)=x_0$. By Corollary \ref{cor:localquot} we know that $f$ is $\Theta$-surjective, hence there exists $w\in \cW$ and a morphism $\psi: \Theta\to \cW$ such that $\psi(1)=w$, $\psi(0)=w_0$, and $\phi=f\circ\psi$. 

Let $\tilde{w}_0$ be the $H$-invariant point in $\spec(A)$ as lifting of $w_0$. Then $\psi$ corresponds to a closed point $\tilde{w}\in \spec(A)$ as lifting of $w$, and a $1$-PS $\lambda: \bG_m\to H$, such that 
$\lim_{t\to 0}\lambda(t)\cdot \tilde{w} =\tilde{w}_0$.
Since $f$ preserves stabilizer groups, we know that the stabilizer group $H_{\tilde{w}}$ is isomorphic to $\Aut(X,D)$. Hence we have $G<H_{\tilde{w}}<H$. By \cite[Lemma 3.1]{New78}, there exists a $H$-equivariant closed embedding $(\spec(A),\tilde{w}_0)\hookrightarrow (V, 0)$ where $V$ is a finite dimensional rational representation of $H$. Then by a result of Luna \cite{luna1975adherences} (see also \cite[Section 6.11, Corollary 1]{VP89}), we know that $\tilde{w}_0\in \overline{N_H (G) \cdot \tilde{w}}$ where $N_H(G)$ is the normalizer subgroup. Since $G$ is finite, the centralizer subgroup $Z_H(G)$ is a finite index subgroup of $N_H(G)$ which implies $\tilde{w}_0\in \overline{Z_H (G) \cdot \tilde{w}}$. Moreover, $Z_H(G)$ is reductive since both $H$ and $G$ are reductive and the ground field is characteristic $0$. Hence there exists a $1$-PS $\lambda': \bG_m\to Z_H(G)$ such that 
$\lim_{t\to 0}\lambda'(t)\cdot \tilde{w} =\tilde{w}_0$. This induces a $G$-equivariant morphism $\psi': \Theta\to \cW$ where the $G$-action on $\Theta$ is trivial. Hence the morphism $\phi'=f\circ\psi': \Theta\to \cM$ provides a $G$-equivariant special test configuration $(\cX',\cD';\cL')$ of $(X,D)$ with central fiber isomorphic to $(\cX_0, \cD_0)$. Clearly $\Fut(\cX',\cD';\cL')=0$ and $(X,D)\not\cong(\cX_0,\cD_0)$, which  contradicts to the $G$-equivariant K-polystability of $(X,D)$. The proof is finished.
\end{proof}

\begin{remark}
It is a key step in the proof to produce a K\"ahler-Einstein metric on $(Y,B+\epsilon \Delta')$ using the result from \cite{li2019uniform}, the proof of which relies on various analytic tools. A purely algebraic proof seems not obvious to us at this point.
\end{remark}


\begin{proof}[Proof of Theorem \ref{thm:finitecover}]
(1) The ``only if'' direction is proven by \cite[Corollary 1.7]{Fuj19}. For the ``if'' direction, assume $(Y,B)$ is K-semistable. Using the same perturbation approach and pulling-back  K\"ahler-Einstein metrics in the proof of Theorem \ref{eq}, we have that $(X,D)$ is K-semistable. This proves part (1).

(2) Assume that $\pi:(X,D)\to (Y,B)$ is Galois with group $G<\Aut(X,D)$. We first treat the ``only if'' direction. Suppose $(X,D)$ is K-polystable. Then by part (1) we know that $(Y,B)$ is K-semistable. Let $(\cY,\cB)$ be a special test configuration of $(Y,B)$ such that $\Fut(\cY,\cB)=0$. Clearly it suffices to show that $(\cY_0,\cB_0)\cong (Y,B)$. By Proposition \ref{prop:specialquot},  a finite base change $(\cY^{(d)},\cB^{(d)})$ of $(\cY,\cB)$ is the $G$-quotient of a $G$-equivariant special test configuration $(\cX,\cD)$ of $(X,D)$. Since $\Fut(\cX,\cD)=d\cdot \Fut(\cY,\cB)=0$ and $(X,D)$ is K-polystable, we know that $(\cX,\cD)$ is a product test configuration. Let $\lambda:\bG_m\to \Aut(X,D)$ be the 1-PS inducing $(\cX,\cD)$. Then $(\cX,\cD)$ being $G$-equivariant implies that 
$\lim_{t\to 0}\lambda(t)\cdot g\cdot \lambda(t)^{-1}$ exists for any $g\in G$. Since $\Aut(X,D)$ is reductive by \cite{ABHLX19}, by \cite{Ric88} we know that representations of $G$ in $\Aut(X,D)$ has closed conjugacy classes. In particular, the group $G':=\{\lim_{t\to 0}\lambda(t)\cdot g\cdot \lambda(t)^{-1}\mid g\in G\}$ is conjugate to $G$ in $\Aut(X,D)$. As a result, $(\cY_0,\cB_0)\cong (X,D)/G'\cong (X,D)/G\cong (Y,B)$. 

For the ``if'' direction, suppose $(Y,B)$ is K-polystable. Then by part (1) we know that $(X,D)$ is K-semistable. Let $(\cX,\cD)$ be a special test configuration of $(X,D)$ such that $\Fut(\cX,\cD)=0$. According to Proposition \ref{prop:specialquot}, after taking $G$-quotient we obtain a special test configuration $(\cY,\cB)$ of $(Y,B)$ with $\Fut(\cY,\cB)=0$. Hence $(\cY,\cB)$ is a product test configuration induced by a $1$-PS $\lambda: \bG_m\to \Aut(Y,B)$. Since $\pi$ is finite, there exists $d\in \bZ_{>0}$ such that $\lambda^d$ lifts to a $1$-PS $\lambda': \bG_m\to \Aut(X,D)$. By Proposition \ref{prop:specialquot}, we know that $(\cX^{(d)},\cD^{(d)})$ is isomorphic to the product test configuration induced by $\lambda'$ since they have isomorphic quotients. Thus $(\cX,\cD)$ is also a product test configuration. 

(3)
Let $\delta:=\delta(X,D)$ and $\delta':=\delta(Y,B)$. By the computation in \cite[Proof of Corollary 1.7]{Fuj19}, we know that $\delta'\geq \delta$. Now assume $(Y,B)$ is K-unstable, then $\delta'<1$. For any $\epsilon>0$, we pick $\Delta'_\epsilon\in |-K_Y-B|_\Q$ such that $(Y,B+(1-\delta'+\epsilon)\Delta'_\epsilon)$ is K-semistable. Then by part (1) of Theorem \ref{thm:finitecover}, we see that $(X,D+(1-\delta'+\epsilon)\Delta_\epsilon)$ is also K-semistable, where $\Delta_\epsilon=\pi^\ast \Delta'_\epsilon$. Therefore Theorem \ref{perturb} implies that $\delta\geq \delta'-\epsilon$ for any $\epsilon$  and hence $\delta = \delta'$.
\end{proof}

\begin{proof}[Proof of Corollary \ref{cor:tian}]
Recall from \cite{Al} that when $G$ is a finite group we have
\[
\alpha_G(X,D)=\inf\{\lct(X,D;\Delta)\mid \Delta\in |-K_X-D|_{\bQ}\textrm{ is $G$-invariant}\}.
\]
Let $(Y,B)$ be the quotient of $(X,D)$ under the $G$-action. Denote by $\sigma:X\to Y$ the quotient map. Then there is a one-to-one correspondence between $G$-invariant $\bQ$-divisors $\Delta\in |-K_X-D|_{\bQ}$ and $\bQ$-divisors $\Delta'\in |-K_Y-B|_{\bQ}$ via $\Delta=\sigma^*\Delta'$. By \cite[Proposition 5.20]{KM98} we know that $\lct(X,D;\Delta)=\lct(Y,B;\Delta')$. Hence $\alpha_G(X,D)=\alpha(Y,B)$. Thus $\alpha_G(X,D)\geq \frac{n}{n+1}$ (resp. $>\frac{n}{n+1}$) implies $\alpha(Y,B)\geq \frac{n}{n+1}$ (resp. $>\frac{n}{n+1}$), which by \cite{Al, FO16} implies that $(Y,B)$ is K-semistable (resp. uniformly K-stable hence admits weak conical K\"ahler-Einstein metrics by \cite{li2019uniform}). 
Thus the proof is finished by Theorem \ref{thm:finitecover}.
\end{proof}

\section{Applications}\label{sec:appl}

In this section we present some applications of our main theorems.

\begin{thm}\label{thm:cycliccover}
Let $B$ be a hypersurface in $\bP^n$ of degree $2\leq d\leq n$. Let $e\geq 2$ be a positive integer that divides $d$. Let $X$ be the degree $e$ cyclic cover of $\bP^n$ branched along $B$. Then $X$ is K-semistable (resp. K-polystable) if $B$ is K-semistable (resp. K-polystable).
\end{thm}

\begin{proof}
By Theorem \ref{thm:finitecover}, it suffices to show that $(\bP^n, (1-\frac{1}{e})B)$ is K-semistable (resp. K-polystable) if $B$ is K-semistable (resp. K-polystable). We first assume that $B$ is K-semistable. Denote by $r:=\frac{n+1}{d}-1$ so that $-K_{\bP^n}\sim_{\bQ}(1+r) B$. Then by \cite[Proposition 5.3]{LX16} and \cite[Lemma 2.12]{LZ19}, we know that $(\bP^n, (1-\frac{r}{n})B)$ admits a special degeneration to a K-semistable log Fano pair $(C_p(B, \cO_B(d)), (1-\frac{r}{n})B_\infty)$ where $C_p(B, \cO_B(d))$ is the projective cone over $B$ with polarization $\cO_B(d)$, and $B_\infty\subset C_p(B, \cO_B(d))$ is the section at infinity. Thus by \cite[Corollary 4]{BL18a} we know that $(\bP^n, (1-\frac{r}{n})B)$ is K-semistable. Since
\[
0<1-\frac{1}{e}\leq 1-\frac{1}{d}<\frac{(d-1)(n+1)}{dn}=1-\frac{r}{n},
\]
by interpolation results of K-stability (see e.g. \cite[Lemma 2.6]{Der16} or \cite[Proposition 2.13]{ADL19}), we know that $(\bP^n,(1-\frac{1}{e})B)$ is K-semistable. If moreover that $B$ is K-polystable, then by \cite{PT06} we know that $B$ is GIT polystable. Then \cite[Theorem 1.4]{ADL19} implies that $(\bP^n, \epsilon B)$ is K-polystable. Then the same interpolation results imply that $(\bP^n, (1-\frac{1}{e})B)$ is K-polystable. The proof is finished.
\end{proof}

\begin{remark}
Note that similar results to Theorem \ref{thm:cycliccover} are obtained in \cite{AGP06} and \cite{Der16} where they considered cyclic covers of Fano varieties with branch loci being Calabi-Yau or of general type, while our branch loci are Fano.
\end{remark}

The following result characterizes the K-moduli space of cubic fourfolds of cyclic cover type. It also implies Theorem \ref{thm:smcubic}.

\begin{thm}
Let $X\subset \bP^5$ be the cubic fourfold with equation $x_5^3=f(x_0,\cdots,x_4)$. Denote by $B=(f=0)\subset\bP^4$ the ramification locus of the cyclic covering map $X\to\bP^4$ as a cubic threefold. Then 
$X$ is K-semistable (resp. K-polystable) if and only if $B$ is GIT semistable (resp. GIT polystable). In particular, any smooth $X$ is K-stable hence admits K\"ahler-Einstein metrics.
\end{thm}

\begin{proof}
The ``if'' direction follows from Theorem \ref{thm:cycliccover} and \cite{LX17b} where it is shown that $B$ is GIT semistable (resp. GIT polystable) if and only if it is K-semistable (resp. K-polystable). For the ``only if'' direction, assume that $X$ is K-semistable (resp. K-polystable). Then Theorem \ref{thm:finitecover} implies that $(\bP^4, \frac{2}{3}B)$ is K-semistable (resp. K-polystable). By Paul-Tian's criterion \cite[Theorem 1]{PT06} (see also \cite[Theorem 2.22]{ADL19}), we know that  $B$ is GIT semistable (resp. GIT polystable).

When $X$ is smooth, we know that $B$ is smooth hence GIT stable. Thus $X$ is K-polystable by the argument above which implies that $X$ is K-stable since it has finite automorphism group. Then the existence of K\"ahler-Einstein metrics follows from \cite{CDS15, Tia15}.
The proof is finished. 
\end{proof}

\bibliographystyle{alpha}
\bibliography{ref}
\end{document}